\documentclass[a4paper, 11pt]{preprint}

\usepackage{nth}

\usepackage[cachedir=minted, frozencache=true]{minted}
\colorlet{CodeBackground}{black!5!white} %
\setminted{encoding=utf8, mathescape, bgcolor=CodeBackground, fontsize=\footnotesize}
\setminted[macaulay2]{label={\TT{Macaulay2}}}
\setminted[mathematica]{label={\TT{Mathematica}}}

\usepackage[style=numeric, citestyle=numeric-comp, backend=biber, sorting=nyt, maxbibnames=99]{biblatex}
\bibliography{refs}

\usepackage{pgfplots}
\pgfplotsset{compat=1.18}

\usepackage{caption}
\usepackage{wrapfig}

\usepackage{dsfont}
\renewcommand{\BB}[1]{\mathds{#1}}

\newcommand{\ZZ}{\BB Z}
\newcommand{\QQ}{\BB Q}
\newcommand{\RR}{\BB R}

\newcommand{\PD}{\RM{PD}}

\newcommand{\Oh}{\mathcal{O}}
\renewcommand{\Pow}[1]{\mathscr{P}(#1)}

\makeatletter
\newcommand{\dol}[1]{\overline{\dbl@overline{#1}}}
\newcommand{\dbl@overline}[1]{\mathpalette\dbl@@overline{#1}}
\newcommand{\dbl@@overline}[2]{%
  \begingroup
  \sbox\z@{$\m@th#1\overline{#2}$}%
  \ht\z@=\dimexpr\ht\z@-2\dbl@adjust{#1}\relax
  \box\z@
  \ifx#1\scriptstyle\kern-\scriptspace\else
  \ifx#1\scriptscriptstyle\kern-\scriptspace\fi\fi
  \endgroup
}
\newcommand{\dbl@adjust}[1]{%
  \fontdimen8
  \ifx#1\displaystyle\textfont\else
  \ifx#1\textstyle\textfont\else
  \ifx#1\scriptstyle\scriptfont\else
  \scriptscriptfont\fi\fi\fi 3
}
\makeatother

\Crefname{problem}{Problem}{Problems}

\TodoColor{Outline}{blue}

\title[Sharp Inequalities for Principal Minors of Positive Definite Matrices]{%
Sharp Inequalities for Products of Principal Minors\\%
of Positive Definite Matrices}

\author{Tobias Boege}
\address{Department of Mathematics and Statistics, UiT The Arctic University of Norway, Tromsø, Norway}
\email{post@taboege.de}

\author{Ludovick Bouthat}
\address{Département de mathématiques et statistiques, Université Laval, Québec, Canada}
\email{ludovick.bouthat.1@ulaval.ca}

\date{\today}

\subjclass[2020]{%
15A15, %
51M16  %
(primary)
14Q30, %
52A40, %
90C23, %
94A17  %
(secondary)%
}

\keywords{%
Determinantal inequality,
Principal minors,
Entropy region,
Information inequality,
Conditional independence,
Cylindrical algebraic decomposition,
Polynomial optimization,
Ingleton inequality%
}

\usepackage{ragged2e}
\usepackage{marginnote}
\begin{document}

\begin{abstract}
We study sharp inequalities for ratios of products of principal minors of real positive definite matrices. Our main result gives a closed-form solution to a family of nonconvex optimization problems over the positive definite cone. As a special case, we prove that the infimum of the Ingleton ratio over $4\times 4$ positive definite matrices is $16/27$, confirming a conjecture of Hall and Johnson. We also show that the cone of absolutely bounded ratios of products of principal minors is not polyhedral for $n\ge 4$, and that it is not semialgebraic over $\mathbb{Q}$. %
\end{abstract}

\maketitle

\section{Introduction}
\label{sec: intro}

Positive definite matrices and their principal minors are a core topic in matrix analysis. They play important roles in various applications, e.g., as determinantal point processes~\cite{DetPointProc}, for semidefinite and polynomial optimization~\cite{Lasserre}, or for Lyapunov stability in control theory~\cite{MatrixAnalysis}. Determinantal inequalities among the principal minors naturally constrain well-studied problems including quantum entanglement~\cite{Horodecki,LogDetEntangle}, the symmetric principal minor assignment problem~\cite{Oeding}, and positive definite matrix completion~\cite{PDCompletion}.

In this paper, we consider a class of constrained optimization problems involving products of principal minors over subsets of $\PD_n$, the cone of real positive definite $n \times n$ matrices. The~rows and columns of these matrices are indexed by the set $N = \{1, \dots, n\}$. For a subset $K \subseteq N$, the principal $K \times K$ minor of $\Sigma \in \PD_n$ is denoted by the bracket $\Sigma[K]$. We omit~$\Sigma$ if it is immaterial or clear from context, and we write $[123]$ instead of $[\{1,2,3\}]$ for brevity. For a set function $\alpha\colon \Pow{N}\to \RR$, let $[\alpha]\colon \PD_n \to \RR$ be the functional $\Sigma \mapsto \Sigma[\alpha] \defas \prod_{K \subseteq N} \Sigma[K]^{\alpha(K)}$, which evaluates the product of the principal minors of $\Sigma$ raised to their respective powers as given by $\alpha$.
We can now state our objective formally:
\begin{problem} \label{prob: main}
Given $\alpha, \beta_1,\dots, \beta_r \colon \Pow{N} \hspace{-1pt}\to\hspace{-1pt} \RR$, minimize $\Sigma[\alpha]$ over $\CC M_{\!\beta} := \bigl\{ \Sigma \hspace{-1pt}\in\hspace{-1pt} \PD_n : \Sigma[\beta_i] = 1 \bigr\}$.
\end{problem}
\noindent%
Since all principal minors are positive, the objective function is well-defined over the domain $\CC M_{\!\beta}$ and the infimum is nonnegative. As the set $\PD_n$ is open, the infimum need not be attained.

We have formulated \Cref{prob: main} broadly because several questions considered in this paper are most naturally viewed as different faces of the same optimization problem. In particular, the validity of \emph{Gaussian information inequalities}, the validity of \emph{conditional information inequalities}, and the search for sharp constants in determinantal inequalities all reduce to instances of \Cref{prob: main}. We state these problems separately below in order to keep track of their different motivations, even though most results in the paper are proved using a matrix-analytic framework.

Although one source of motivation comes from information theory, the results and proofs are formulated in terms of positive definite matrices and their principal minors. \Cref{sec: info} reviews the relevant information-theoretic background and explains why the particular functionals studied here are natural. Readers primarily interested in matrix theory may read this section as motivation: the main theorems and their proofs only require the determinantal formulations stated there.

\subsection*{Determinantal inequalities}

\Cref{prob: main} is intimately connected to determinantal inequalities for principal minors. %
Among the many prior works on this topic, we highlight two which are closest to our approach and inspired some of our investigations. Following Hall and Johnson~\cite{HJ}, a functional $[\alpha]$ is \emph{bounded} if there exists a constant $c>0$, independent of~$\Sigma$, such that $\Sigma[\alpha] \ge c$ for all $\Sigma \in \PD_n$. (To be precise, Hall and Johnson consider functionals $[\alpha]$ which are bounded from above, but this is of course equivalent to $[-\alpha] = [\alpha]^{-1}$ being bounded away from~$0$.) The functional $\alpha$ is \emph{absolutely bounded} if $c = 1$ is a lower bound. Note that if $\alpha$ and $\alpha'$ are (absolutely) bounded and $\lambda \ge 0$, then $\alpha + \lambda \alpha'$ is also (absolutely) bounded. Thus, the collections $\CC B_n$ and $\CC A_n$ of bounded and absolutely bounded functionals, respectively, are convex cones in $\BB R^{2^n}$.

The paper by Johnson and Barrett \cite{JB} is dedicated to inequalities among two products of principal minors and contains a characterization of absolutely bounded $\alpha$ whenever all but six entries of $\alpha$ are zero. The criterion boils down to a sequence of applications of the Koteljanskii inequality, stating that for any $K, L \subseteq N$,
\begin{equation}
  \label{eq: kotel}
  [K][L] \ge [K \cup L] [K \cap L].
\end{equation}
It follows from these results that the cone $\CC A_3$ is \emph{polyhedral}, i.e., the convex hull of finitely many extreme rays (namely, given by the Koteljanskii ratios).
Hall and Johnson~\cite{HJ} investigate, more generally, bounded ratios of products of principal minors. Generalizing the Koteljanskii machinery to inequalities derived from the rank functions of linear matroids, they show that $\CC B_4$ is polyhedral and conjecture that this is true in higher dimensions as well. Our first main result is that the analogous conjecture for $\CC A_n$ fails already at the earliest possible step.

\begin{result} \label{res: A4 non-polyhedral}
The convex cone $\CC A_n$ is not polyhedral for $n \ge 4$.
\end{result}

The proof, which is the main objective of \Cref{sec: cone}, is very much inspired by an analogous result in information theory due to Matúš~\cite{MatusInfinite}. In particular, it comes with an infinite sequence of concrete determinantal inequalities which approaches the boundary of $\CC A_4$ in a genuinely nonlinear way. This gives a concrete obstruction to polyhedrality, rather than only an abstract separation argument. Since $\CC A_n$ is obtained as a linear projection of $\CC A_{n+1}$, the non-polyhedrality of $\CC A_4$ implies non-polyhedrality in higher dimensions.

\subsection*{The Ingleton and GMM functionals}

The functional $\varphi$ given by
\begin{equation}
  \label{eq: ingleton}
  [\varphi] = \frac{[13][14][23][24][34]}{[3][4][12][134][234]}
\end{equation}
appears (permuted and reciprocal) in \cite[Example~5.2]{JB} as a sporadic example that the necessary conditions for absolute boundedness of Johnson and Barrett are not sufficient. Hall and Johnson \cite[Example~2]{HJ} prove that $[\varphi] \ge \sfrac14$ and conjecture that its true infimum is $\sfrac{16}{27}$. This is our second main result.

\begin{result} \label{res: 1627}
The infimum of $[\varphi]$ over $\PD_4$ is $\sfrac{16}{27}$ and it is not attained.
\end{result}

Despite its ad-hoc appearance in the context of \cite{JB}, this functional is very well known. It~is a celebrated result of Ingleton \cite{Ingleton} that if $r\colon \Pow{N} \to \ZZ$ is the rank function of a linear matroid, i.e., it is the dimension of the span of some family of vectors indexed by $K\subseteq\Pow{N}$, then it satisfies the \emph{Ingleton inequality} %
\[
\langle r, \varphi\rangle = r(13)+r(14)+r(23)+r(24)+r(34)-r(3)-r(4)-r(12)-r(134)-r(234) \geq 0.
\]
The significance of this inequality in information theory, namely in the Shannon setting of discrete random variables, was highlighted by Matúš and Studený~\cite{MatusI}. While Shannon entropy profiles can violate the Ingleton inequality, the study of sufficient conditions on the distribution under which the inequality does hold turned out to be extremely fruitful; see~\cite{CondIngleton,Milan1,AndreiSlava}.

Our techniques for proving the $\sfrac{16}{27}$ bound extend to a more complicated family of functionals
\begin{equation}
  \label{eq: gmm}
  [\varphi_k] \defas [\varphi] \cdot \left(\frac{[13][34]}{[3][134]} \cdot \frac{[14][34]}{[4][134]} \cdot \frac{[23][34]}{[3][234]} \cdot \frac{[24][34]}{[4][234]}\right)^k, \text{ for real $k \ge 0$},
\end{equation}
which originates from work of Gómez, Mejía, and Montoya \cite{GMM}. The functions $\varphi_k$ are the centerpiece in their approach to proving that the so-called \emph{almost-entropic region} in the Shannon setting is not semialgebraic, i.e., it is not described by a finite system of polynomials equations and inequalities. Not enough is known about $\varphi_k$ yet to reach this conclusion.

Semialgebraicity is a useful geometric ``tameness'' condition which has algorithmic implications. By definition, to test whether $\alpha \in \CC A_4$, one has to prove the universal inequality $\Sigma[\alpha] \ge 1$ for all $\Sigma \in \PD_4$. If $\CC A_4$ were semialgebraic, then testing $\alpha \in \CC A_4$ would be as simple as plugging $\alpha$ into finitely many polynomials and evaluating their signs.
Therefore, the semialgebraicity problem is a major open question in the theory of information inequalities and a benchmark problem to test new information-theoretic constructions against \cite{CsirmazExploring,AndreiSlava}; we~discuss the connection of this problem to the family $\varphi_k$ more in \Cref{sec: info}. Moreover, in \Cref{sec: 1627}, we completely solve the Gómez--Mejía--Montoya mystery for $\PD_4$~matrices.

\begin{result} \label{res: inf varphi_k}
The infimum of $[\varphi_k]$ over $\PD_4$ is given by
\[
  \begin{cases}
    \frac{(4k+4)^{4k+4}}{16 (4k+3)^{4k+3}}, & \text{if $0 \le k < k_*$}, \\
    1, & \text{if $k \ge k_*$},
  \end{cases}
\]
where $k_* \approx 0.59829$ is the unique positive root of the equation $(4k+4)^{4k+4} = 16 (4k+3)^{4k+3}$.
\end{result}

The proof of this result is long but remarkably elementary. Moreover, we note that its most technical part has
also been formalized in Lean.

At $k = 0$, \Cref{res: inf varphi_k} recovers \Cref{res: 1627}. For $k=1 > k_*$ we find a functional which is absolutely bounded but not a product of Koteljanskii ratios, answering another question of \cite[Section~4]{HJ}. This closed-form solution to the optimization problem for $[\varphi_k]$ yields further insights into the geometry of the convex cone~$\CC A_4$.

\begin{result} \label{res: A4 non-semialg}
The value $k_*$ is transcendental and therefore $\CC A_n$, $n \geq 4$, is not semialgebraic over~$\QQ$.
\end{result}

There is a subtlety to the interpretation of our result: we only show that $\CC A_4$ is not semialgebraic over~$\QQ$. It is still open whether $\CC A_4$ is semialgebraic over an extension field $\QQ(\tau_1, \dots, \tau_m)$, where $\tau_1, \dots, \tau_m$ are finitely many transcendental numbers; in~particular we cannot say whether $\CC A_4$ is semialgebraic over~$\QQ(k_*)$. To rule out such a representation one would have to find a whole transcendental curve of valid information inequalities whose points approach the boundary of~$\CC A_4$ arbitrarily well.
For the non-polyhedrality result (\Cref{res: A4 non-polyhedral}) we exhibit a non-linear (quadratic) curve at the boundary of~$\CC A_4$. This rules out a description by finitely many linear inequalities even when their coefficients can be drawn from~$\RR$. This is precisely why \Cref{res: A4 non-semialg} does not supersede \Cref{res: A4 non-polyhedral}.

\subsection*{Organization of the paper}
In \Cref{sec: info}, we recall the connection between Gaussian information inequalities and determinantal inequalities for principal minors. This section also introduces conditional inequalities and the GMM problem, which motivate the definition of the family $\varphi_k$. In \Cref{sec: cone}, we use copy-type arguments and essential conditionality to prove that $\CC A_n$ is not polyhedral for $n \ge 4$. Finally, \Cref{sec: 1627} is devoted to the sharp optimization problem for the Ingleton and GMM functionals; this proves the $\sfrac{16}{27}$ bound, determines the threshold $k_*$, and yields the non-semialgebraicity result over~$\QQ$.

\subsection*{Computational aspects}

In general, neither the objective function, nor the constraints in \Cref{prob: main} are convex, and local optima need not be global. However, it is an instance of \emph{polynomial programming}. %
By standard results from real algebraic geometry (see \cite{RAGOpt}), the infimum is a computable algebraic number. However, in practice, these problems are very challenging since we seek \emph{sharp} lower bounds on rational functions whose numerator and denominator usually feature high degrees and many terms. For instance, the Ingleton functional $[\varphi]$ has numerator and denominator both of degree 10, with 32 and 48 terms, respectively. In this case, the infimum is not attained, and approaching it requires approaching the boundary of the $\PD_4$ cone, where some of the principal minors become zero. This, in turn, creates numerical instabilities.

For all of our results, we provide a classical analytic proof. Nonetheless, we also use in parallel an entirely robust but computationally expensive symbolic approach to \Cref{prob: main} via \emph{cylindrical algebraic decomposition (CAD)}. A delightful hands-on introduction to this tool is given in \cite{Kauers}. The problems we discuss in this article appear to be just beyond the verge of intractability of naïve applications of CAD as it is currently implemented in Wolfram {Mathematica}~\cite{Mathematica}.

Nevertheless, computation is used in various places in this paper to identify specific examples. In addition to CAD, we use power series expansions of rational or analytic functions to construct interesting curves of positive definite matrices.
The source code for all these computations, including the Lean formalization of \Cref{res: inf varphi_k}, is available from this Zenodo record:
\begin{center}
  \url{https://doi.org/10.5281/zenodo.20785034}.
\end{center}

\section{From information theory to polynomial optimization}
\label{sec: info}

This section introduces concepts from information theory and associated decision problems. We~explain how these tasks relate to the matrix-analytic \Cref{prob: main}. In the following sections, we freely use these concepts to motivate the study of concrete instances of this problem, and to transfer information-theoretic techniques to obtain results of general interest for matrix theorists.

\subsection{Information inequalities}

The differential entropy of an $n$-variate Gaussian random vector $X$ with positive definite covariance matrix $\Sigma$ is $H(X) := \frac{n}{2} \log(2 \pi e) + \frac12 \log \lvert \Sigma \rvert$. The \emph{entropy profile} $h_X\colon \Pow{N} \to \BB R$ of $X$ records the entropies of all its marginals, including the entire distribution and the value $0$ for the marginal corresponding to the empty subvector:
\begin{equation}
  h_X(K) = H(X_K) = \frac{|K|}{2} \log(2 \pi e) + \frac12 \log \Sigma[K].
\end{equation}
For fixed $n$, let $\gamma_n^* \subseteq \BB R^{2^n}$ be the set of all such entropy profiles as $\Sigma$ varies in $\PD_n$.
An~\emph{information functional} is any linear functional of the entropy profile. It takes the form of a scalar product $\langle \alpha, h_X\rangle = \sum_{K \subseteq N} \alpha(K) h_X(K)$ for some $\alpha\colon \Pow{N} \to \BB R$. %
An information functional such that $\langle \alpha, h\rangle \ge 0$ for all $h \in \gamma_n^*$ represents an \emph{information inequality}. The information inequalities for a fixed $n$ form a convex cone. They encode universal information-theoretic properties of Gaussian random variables, which are essential tools in certifying the optimality of solutions to coding problems; see~\cite[Chapter~12]{Yeung}.

Determining if a functional $\alpha$ is a valid information inequality is a fundamental task which received considerable attention in the Shannon setting of discrete random variables. To illustrate this, it shall suffice to mention the catalogue of Dougherty, Freiling and Zeger \cite{DFZ11} containing more than 200 valid inequalities among four discrete random variables, and several infinite families. The analogous problem for Gaussian random variables has received less attention, although its connection with determinantal inequalities is well known~\cite{CoverThomasDet,InfoDet,GaussianEntropy}. In this setting, validity can be reformulated as a question on products of principal minors, and hence as a special case of \Cref{prob: main}. Let us mention a necessary condition for validity, namely the \emph{balance condition}. This condition appears in \cite[Theorem~1]{ChanBalanced} for arbitrary continuous distributions, and it remains necessary in the Gaussian setting, as it follows by simply considering diagonal covariance matrices~\cite{JB}.

\begin{lemma}[Balance condition]
For $\alpha\colon \Pow{N} \to \BB Z$ to be a valid information inequality on $\gamma_n^*$, it has to be \emph{balanced}: for each $i \in N$, it must satisfy $\sum_{i \in K \subseteq N} \alpha(K) = 0$.
\end{lemma}

By adding up all these equations, it follows that $\sum_{K \subseteq N} |K| \, \alpha(K) = 0$. Hence, for balanced~$\alpha$:
\begin{align*}
  \langle \alpha, h_X \rangle &= \frac12 \log(2\pi e) \sum_{K \subseteq N} |K| \, \alpha(K) + \frac12 \sum_{K \subseteq N} \alpha(K) \log \Sigma[K] = \frac12 \sum_{K \subseteq N} \alpha(K) \log \Sigma[K].
\end{align*}
Thus, the value of a balanced functional depends only on the logarithms of the principal minors. By clearing the factor $\sfrac12$ and taking the exponential, we arrive at the following equivalent formulation of checking the validity $\langle \alpha, h_X\rangle \ge 0$:

\begin{problem}[Validity of information inequalities] \label{prob: valid}
For a balanced $\alpha\colon \Pow{N} \to \BB Z$, decide whether $\Sigma[\alpha] \ge 1$ for all $\Sigma \in \PD_n$.
\end{problem}

Functionals which are bounded from below by~$1$ are called \emph{absolutely bounded} in \cite{HJ}. Hence, proving validity of Gaussian information inequalities is nothing but proving absolute boundedness of the corresponding ratio of products of principal minors.
\Cref{prob: valid} is thus an instance of \Cref{prob: main} where the optimization is unconstrained. %

The~Shannon version of the validity problem, where the random variables are discrete rather than Gaussian, is conjectured to be undecidable. For the related problem of membership testing in the dual cone, undecidability has recently been established by Yashfe~\cite{Yashfe}. The connection between the Shannon and Gaussian settings is provided by a result of Chan \cite[Theorem~2]{ChanBalanced}, which shows that balanced information inequalities valid for discrete random variables are also valid for continuous random variables.

\begin{theorem}[Chan's transfer principle] \label{thm: Chan transfer}
Let $\alpha\colon \Pow{N} \to \RR$ be balanced. It is a valid information inequality for discrete random variables if and only if it is valid for continuous random variables.
\end{theorem}

It follows that all balanced and valid information inequalities in the Shannon setting must be valid for Gaussians. Contrapositively, this gives an approach to disprove information inequalities in the Shannon setting: if there exists $\Sigma \in \PD_n$ such that $\Sigma[\alpha] < 1$, then $\alpha$ is not a valid information inequality for discrete random variables (provided that $\alpha$ is balanced).

\subsection{Extreme position constraints}

If $[\beta] \ge 1$ is a valid information inequality, then the set $\CC M_{\!\beta} = \Set{ \Sigma \in \PD_n : \Sigma[\beta] = 1 }$ is of interest. It parametrizes those Gaussian distributions which are in ``extreme position'' with respect to the functional~$\beta$. Constraints of this form describe various information-theoretic special situations, such as independence of random variables. Under such assumptions, stronger inequalities may hold, which are called \emph{conditional information inequalities}. This motivates the shape of the constraints in \Cref{prob: main}.

\begin{problem}[\hspace{-1pt}Validity of conditional information inequalities] \label{prob: cond valid}
\hspace{-.5pt}Let $\alpha, \beta_1, \dots, \beta_r\colon \hspace{-1pt}\Pow{N} \hspace{-1pt}\to\hspace{-1pt} \ZZ$ be balanced. Decide whether $\Sigma[\alpha] \ge 1$ for all $\Sigma \in \PD_n$ satisfying $\Sigma[\beta_1] = \dots = \Sigma[\beta_r] = 1$.
\end{problem}

For any $\CC M \subseteq \PD_n$ we write $\CC M \models [\alpha] \ge 1$ if the inequality holds for all $\Sigma \in \CC M$, i.e., if $[\alpha] \ge 1$ is valid conditionally on the set~$\CC M$. If $\CC M = \Set{\Sigma \in \PD_n : \Sigma[\beta_1] = \dots = \Sigma[\beta_r] = 1}$, then this may also be written as an implication among determinantal constraints: $[\beta_1] = \dots = [\beta_r] = 1 \implies [\alpha] \ge 1$.
This~is especially meaningful if the functionals $\beta_1, \dots, \beta_r$ are themselves valid information inequalities, so that the constraints correspond to extreme position assumptions.

The Koteljanskii inequality \eqref{eq: kotel} is of special importance: it shows that the \emph{Koteljanskii ratios} $$[I:J\mid K] \defas \frac{[I\cup K][J\cup K]}{[I\cup J\cup K][K]}$$ are absolutely bounded. We may assume without loss of generality that $I, J, K \subseteq N$ are pairwise disjoint. In information-theoretic language, this shows that the \emph{conditional mutual information} is nonnegative. These elementary information inequalities (and their nonnegative linear combinations) are also known as \emph{Shannon inequalities}.
The corresponding extreme position $[I \cup K][J \cup K] = [I\cup J\cup K][K]$ characterizes the \emph{conditional independence} $\CI{X_I,X_J|X_K}$, a relation commonly used in information theory and statistical modeling~\cite{Studeny}. Therefore, conditional independence assumptions can be encoded in the constraints of \Cref{prob: main}, and this will indeed be our main use for them.

The Ingleton functional $[\varphi]$ from \eqref{eq: ingleton} is not absolutely bounded and is therefore not a valid information inequality for Gaussian or discrete random variables. Conditions on discrete random variables, especially of conditional independence type, which make Ingleton valid have been thoroughly studied; cf.~\cite{ZhangYeung,KR,MatusI,CondIngleton,Milan1}. One reason for this interest is the fact that all valid conditional independence implications on four discrete random variables follow from the conditional Ingleton inequalities~\cite{CondIngleton}. In this sense, they completely resolve a low-dimensional instance of a problem that is undecidable in general \cite{KuehneYashfe,LiUndecidable}. We shall return to the Ingleton inequality in \Cref{sec: 1627} where we will prove \Cref{res: inf varphi_k}.

\subsection{Relaxations and stability}

The concept of an \emph{essentially conditional} information inequality was introduced by Kaced and Romashchenko~\cite{KR}, together with a proof of their existence. The main idea is that if
\begin{equation}
  \label{eq: cond ineq}
  [\beta_1] = \dots = [\beta_r] = 1 ~\implies~ [\alpha] \ge 1
\end{equation}
is a valid conditional information inequality, then it might follow from an unconditional inequality of the form
\begin{equation}
  \label{eq: uncond ineq 1}
  \PD_n \models [\alpha] \ge [\lambda_1 \beta_1 + \dots + \lambda_r \beta_r] = [\beta_1]^{\lambda_1} \cdots [\beta_r]^{\lambda_r},
\end{equation}
for some positive real numbers $\lambda_i$. Indeed,~\eqref{eq: cond ineq}~is a straightforward consequence of~\eqref{eq: uncond ineq 1} and thus, one wishes to find this stronger version of the original inequality. If such multipliers exist, then the conditional inequality \eqref{eq: cond ineq} is called \emph{unconditional}, otherwise it is \emph{essentially conditional}. Note that \eqref{eq: uncond ineq 1} remains valid if each $\lambda_i$ is replaced by a larger number. Hence, we may assume in particular that all $\lambda_i$, if they exist, are equal and arbitrarily large. By taking the $\lambda$-th root of the inequality, we arrive at the following equivalent formulation.

\begin{problem}[Essential conditionality] \label{prob: ess cond}
Let $[\beta_1] = \dots = [\beta_r] = 1 \implies [\alpha] \ge 1$ be a valid conditional information inequality. Decide if there exists $\mu > 0$ such that $\PD_n \models [\alpha]^\mu \ge \prod_{i=1}^r [\beta_i]$.
\end{problem}

Consider a valid, essentially conditional inequality for Gaussians. By \Cref{thm: Chan transfer}, this inequality is either invalid or essentially conditional when interpreted for discrete random variables. %
The Gómez--Mejía--Montoya (GMM) problem asks whether the valid inequality
\begin{equation}
  \label{eq: GMM}
  [1:3\mid 4] = [1:4\mid 3] = [2:3\mid 4] = [2:4\mid 3] = 1 \implies [\varphi] \ge 1
\end{equation}
is essentially conditional for four discrete random variables. This question originates from~\cite{GMM}, where it is shown that an affirmative answer implies that the convex cone of valid information inequalities for four discrete random variables is not semialgebraic.
Thanks to Chan's transfer principle, their proof translates to the Gaussian setting.
If \eqref{eq: GMM} is essentially conditional for Gaussians, it must be essentially conditional for discrete random variables as well, and thus both of their cones of information inequalities would not be semialgebraic for $n \ge 4$.

Thus, the Gaussian version of the GMM problem provides a testing ground for the same strategy: if the conditional Ingleton inequality were essentially conditional for Gaussians, then the Gaussian cone would inherit the same kind of non-semialgebraicity phenomenon. Unfortunately, we show in \Cref{sec: 1627} that \eqref{eq: GMM} is in fact unconditional for Gaussians. Nonetheless, the analysis reveals a different reason for why $\CC A_4$ is not semialgebraic over~$\QQ$, but this has no bearing on the GMM problem for discrete random variables.

\begin{remark}[On computational complexity]
\Cref{prob: valid,prob: cond valid} are instances of \Cref{prob: main}, and therefore decidable.
When the parameter $\mu$ is fixed and rational, then \Cref{prob: ess cond} is also decidable. If $\mu$ is not fixed, this problem concerns the truth of $\exists\forall$ formulas in the theory of $\RR_{\exp}$, the field of real numbers with the exponential function \cite{MarkerExp}. Whether this theory is decidable is still an open problem linked to a weak version of Schanuel's conjecture; see~\cite{RealExponential,SchanuelAlgo}. %
\end{remark}

To summarize, \Cref{prob: valid} identifies absolute boundedness with the validity of Gaussian information inequalities, while \Cref{prob: cond valid} explains the role of the constraints in \Cref{prob: main}. The notion of essential conditionality will be used in \Cref{sec: cone} to prove non-polyhedrality, whereas the GMM family $\varphi_k$ will be studied directly in \Cref{sec: 1627} through sharp optimization over~$\PD_4$.

\section{The cone of Gaussian information inequalities}
\label{sec: cone}

The goal of this section is to prove \Cref{res: A4 non-polyhedral}, namely that the cone $\CC A_4$ of absolutely bounded ratios of principal minors is not polyhedral. The proof has three steps. First, we recall the so-called \emph{Copy lemma} in the context of positive definite matrices. This lemma has been used to prove several different inequalities, including an infinite family of inequalities due to Matúš, which is presented below. In particular, the well-known \emph{Zhang--Yeung inequality} is a special case of these inequalities. Although it is not necessary in itself, we demonstrate the use of the Copy~lemma by proving this inequality.

Second, we show that Matúš's family of inequalities implies a conditional version of the Ingleton inequality. Using this fact, we then prove that this conditional inequality is essentially conditional. A standard argument of Kaced and Romashchenko finally converts essential conditionality into non-polyhedrality of $\CC A_4$, as desired.

\subsection{Copy constructions and the Matúš family}

\begin{lemma}[Copy lemma] \label{lemma: copy}
Let $K,L\subseteq N$ be two disjoint sets and let $\Sigma \in \PD_{|K \cup L|}$. Let $K'$ be any set disjoint from $K \cup L$ with $|K'| = |K|$. Then there exists a unique $\Sigma' \in \PD_{|K \cup L \cup K'|}$ such that $\Sigma'_{K \cup L} = \Sigma = \Sigma'_{K'\cup L}$ and $\Sigma'[K:K'\mid L] = 1$.
\end{lemma}

This result embeds every positive definite matrix into a larger matrix which satisfies an extreme position relation depending on the partition~$K \cup L$. The Koteljanskii inequalities for $\Sigma'$ are linear inequalities in the logarithmic principal minors $(\log\Sigma'[J] : J \subseteq K \cup L \cup K')$. The~Copy~lemma also guarantees linear equations in these coordinates. Hence, linear programming can be used to find linear inequalities which constrain the logarithmic principal minors of $\Sigma'$ and hence those of~$\Sigma$; see~\cite{DFZ11,Undiscovered}. This lift-and-project technique finds inequalities for $n \times n$ matrices which are not provable using $n \times n$ Koteljanskii inequalities.
The Zhang--Yeung inequality from~\cite{ZhangYeung} is a famous example; it~was independently proved for Gaussians by Lněnička \cite{Lnenicka}, who also characterized the matrices in extreme position with respect to this inequality. We present here an entirely mechanical proof of this inequality which was found by a linear programming solver. The complete source code of this computation is contained in our Zenodo~record (\textsf{zhang-yeung-lp.nb}).

\begin{example}[Zhang--Yeung inequality]
We want to prove the following inequality:
\begin{equation}
  \label{eq: zy}
  \PD_4 \models [\varphi] \cdot [1:4\mid 3] \cdot [1:3\mid 4] \cdot [3:4\mid 1] \ge 1.
\end{equation}
Let $\Sigma \in \PD_4$ be arbitrary. Its rows and columns are indexed by $\{1,2,3,4\}$. Invoke \Cref{lemma: copy} with $L = \{3,4\}$ to create copies $1', 2'$ of the indices $1, 2$ while keeping $3$ and $4$ fixed. This yields $\Sigma' \in \PD_6$ indexed by $\{1,2,3,4,1',2'\}$. Koteljanskii ratios are absolutely bounded, so the product of the following thirteen Koteljanskii ratio powers is absolutely bounded:
\begin{equation}
  \label{eq: zy1}
  \begin{gathered}
  [1:2\mid 1'],\;\; [1:1'\mid 3],\;\; [1:1'\mid 4],\;\; [1:1'\mid 234],\;\; [1:2'\mid 1'34]^3,\;\; [1':2\mid 3],\;\; [1':2\mid 4], \\
  [1':2\mid 134],\;\; [2:2'\mid 11'34]^3,\;\; [3:4\mid 11'],\;\; [3:4\mid 1'2],\;\; [1':3\mid 124],\;\; [1':4\mid 12].
  \end{gathered}
\end{equation}
Due to the Copy lemma, the following expressions all evaluate to one on $\Sigma'$:
\begin{equation}
  \label{eq: zy2}
  \left(\frac{[34][11'22'34]}{[1234][1'2'34]}\right)^{\!3},\;\; \frac{[1']}{[1]},\;\; \left(\frac{[13]}{[1'3]}\right)^{\!2},\;\; \left(\frac{[14]}{[1'4]}\right)^{\!2},\;\; \left(\frac{[1'34]}{[134]}\right)^{\!3}.
\end{equation}
Multiplying all terms of \eqref{eq: zy1} and \eqref{eq: zy2} together causes cancellations, all principal minors involving the copy variables $1'$ and $2'$ vanish and the Zhang--Yeung inequality~\eqref{eq: zy} remains. Since this expression is $\ge 1$ on $\Sigma'$ but only involves principal minors of $\Sigma$, it must also be $\ge 1$ on $\Sigma$.
\end{example}

The preceding example has an infinite analogue. Indeed, in order to prove that the convex cone of information inequalities for discrete random variables is not polyhedral, Matúš \cite{MatusInfinite} has generalized the Zhang--Yeung inequality through inductive use of the Copy~lemma. The resulting family of inequalities holds for discrete random variables and each of them is balanced. Hence, Chan's transfer principle implies that they also hold for Gaussians. Thus we obtain:

\begin{theorem}[{\cite[Corollary~3]{MatusInfinite}}] \label{thm: Matus infinite}
For every $s \geq 0$ the following holds:
\begin{equation}
  \label{eq: Matus}
  \PD_4 \models [\varphi]^s \cdot [1:4\mid 3] \cdot \big( [1:3\mid 4] \cdot [3:4\mid 1] \big)^{\frac{s(s+1)}{2}} \ge 1.
\end{equation}
\end{theorem}
Each functional is a point in $\CC A_4 \subseteq \RR^{16}$ and as $s \geq 0$ varies, it traces out a quadratic curve inside~$\CC A_4$. To show that $\CC A_4$ is not polyhedral, it suffices to show that this curve approaches the boundary of $\CC A_4$ arbitrarily well, so that the boundary cannot be linear. Matúš has done this in the Shannon setting by constructing a suitable family of discrete probability distributions, but his proof does not transfer to the Gaussian setting.

\subsection{Essential conditionality}

Since, Matúš's proof does not transfer to the Gaussian setting, we instead follow an insight of Kaced and Romashchenko \cite{KR}, who explained that this is in fact a question about essential conditionality (\Cref{prob: ess cond}). Namely, the sequence of unconditional inequalities from \Cref{thm: Matus infinite} implies the conditional inequality
\begin{equation}
  \label{eq: Matus cond}
  [1:3\mid 4] = [3:4\mid 1] = 1 \implies [\varphi] \ge 1.
\end{equation}
Because if not then there exists $\Sigma \in \PD_4$ with $\Sigma[1:3\mid 4] = \Sigma[3:4\mid 1] = 1$ and $\Sigma[\varphi] < 1$. For any $s > 0$, we still have \eqref{eq: Matus}, hence $\Sigma[\varphi]^s \cdot \Sigma[1:4\mid 3] \ge 1$. However, this is absurd since the product converges to zero as $s \to \infty$.

If the valid conditional inequality \eqref{eq: Matus cond} is essentially conditional, it cannot be ``linearized'' via Lagrange multipliers, implying that $\CC A_4$ is non-polyhedral. Before explaining the details of this argument, we prove essential conditionality.

\begin{lemma} \label{lemma: Matus ess cond}
The conditional inequality \eqref{eq: Matus cond} is essentially conditional.
\end{lemma}

\begin{proof}
The goal is to construct matrices $\Sigma(\mu) \in \PD_4$ which satisfy
\begin{equation}
  \label{eq: Matus uncond}
  \Sigma(\mu)[\varphi]^\mu \cdot \Sigma(\mu)[1:3\mid 4] \cdot \Sigma(\mu)[3:4\mid 1] < 1
\end{equation}
for every $\mu > 0$. Clearly, it is sufficient to find these matrices for arbitrarily small $\mu>0$. The derivation of such a family of matrices was aided by symbolic computation in {Mathematica}. The protocol is available in our Zenodo record (\textsf{non-polyhedral.nb}). Note that once the family is found, the proof's verification requires only elementary power series expansions of analytic functions.

We fix the notation
\[
  \Sigma(\mu) = \begin{pmatrix}
  1 & a & b & c \\ a & 1 & d & e \\ b & d & 1 & f \\ c & e & f & 1
  \end{pmatrix},
\]
where each entry is implicitly a function of~$\mu$. To achieve \eqref{eq: Matus uncond}, it is helpful to approach the locus where $[1:3\mid 4] = [3:4\mid 1] = 1$. This locus is characterized by the equations $b = cf$ and $f = bc$. Since, by positive definiteness, every off-diagonal entry lies in the interval $(-1, 1)$, all solutions to these equations have $b = f = 0$ which implies~$[\varphi] \ge 1$. However, the equations can be approximately satisfied in a different way which allows $[\varphi] < 1$: by letting $c \to 1$. Considerations and experiments of this kind lead to the following ansatz:
\begin{gather*}
  b = c f, \quad c = 1 - \eps, \quad f = \eps^2, \\
  d = \delta, \quad e = 1 - \delta, \quad a = b^2,
\end{gather*}
where $\eps > 0$ is sufficiently small and $\delta$ is an additional parameter. The inequality \eqref{eq: Matus uncond} with $\mu = 1$ subject to $\Sigma(\eps, \delta) \in \PD_4$ is a semialgebraic problem in two parameters which {Mathematica}'s \mintinline{mathematica}{GenericCylindricalDecomposition} easily solves. It can be satisfied provided that $\eps > 0$ is small enough and $\delta$ is between the \nth{3} and \nth{4} real roots of the polynomial
\begin{align*}
    p(x) &= \varepsilon(1-\varepsilon^{2})(1+\varepsilon^{2}(1-\varepsilon)^{2})
\left(2x(1-x^{2})-\varepsilon\right)
\\
&\;\;\;\;\;+x^{2}(1-x^{2})\!\left(
-1+2\varepsilon^{2}-6\varepsilon^{3}+2\varepsilon^{4}+2\varepsilon^{5}-\varepsilon^{6}
+(1-\varepsilon^{2}(1-\varepsilon)^{2})x^{2}
\right).
\end{align*}
These roots both converge to~$1$ as $\eps \to 0$. Using \mintinline{mathematica}{AsymptoticSolve} in {Mathematica}, we find power series representations of the roots of $p$ near $x=1$ as functions of $\eps$ around~$0$. It turns out that the polynomial $\delta = 1 - \frac12 \eps$ stays between these two roots. This defines a curve $\Sigma(\eps) \in \PD_4$ such that as $\eps \to 0$ the inequality \eqref{eq: Matus uncond} with $\mu = 1$ is eventually true. It remains to couple $\eps$ to $\mu$ and to check if the resulting curve proves \eqref{eq: Matus uncond} for arbitrarily small~$\mu$. With $\eps = \mu^2$ one can check~that
\begin{align*}
  \Sigma(\mu)[12] &= 1 - \mu^8 + \Oh(\mu^{10}), \\
  \Sigma(\mu)[123] &= \mu^2 - \frac54 \mu^4 + \Oh(\mu^{6}), \\
  \Sigma(\mu)[1234] &= \mu^4 + 3\mu^6 + \Oh(\mu^{8}), \\
  \Sigma(\mu)[\varphi]^\mu \cdot \Sigma(\mu)[1:3\mid 4] \cdot \Sigma(\mu)[3:4\mid 1] &= 1 - \mu^5 + \Oh(\mu^6). \qquad\qquad
\end{align*}
Hence for every sufficiently small $\mu > 0$ the matrix $\Sigma(\mu)$ is positive definite and satisfies~\eqref{eq: Matus uncond}.
\end{proof}

\subsection{Non-polyhedrality}

\begin{theorem} \label{thm: non polyhedral}
The convex cone $\CC A_4$ is not polyhedral. Hence, $\CC A_n$ for $n \ge 4$ is not polyhedral.
\end{theorem}

\begin{proof}
The general result follows from the $n = 4$ instance because $\CC A_n$ is a coordinate projection of $\CC A_{n+1}$ and polyhedral cones are closed under linear projections~\cite[Theorem~1.4]{Ziegler}.

The following argument connecting essential conditionality to non-polyhedrality is due to Kaced and Romashchenko \cite[Section~VI]{KR}.
The argument is most natural in the ambient vector space of functions $\Pow{\{1,2,3,4\}} \to \RR$, which we identify with~$\RR^{16}$. The cone $\CC A_4$ is dual to $\gamma_4^*$; so~the dual of $\CC A_4$ is the convex conic hull of $\gamma_4^*$ which we denote by~$\CC C_4$.
We need one more piece of notation. Let $\Delta_{ij|k}\colon \Pow{N} \to \RR$ be the set function corresponding to the Koteljanskii ratio $[\Delta_{ij|k}] = \frac{[ik][jk]}{[ijk][k]}$. The inequalities $\Sigma[1:3\mid 4] \ge 1$ and $\Sigma[3:4\mid 1] \ge 1$ for all $\Sigma \in \PD_4$ translate into $\langle \Delta_{13|4}, h\rangle \ge 0$ and $\langle \Delta_{34|1}, h\rangle \ge 0$ for all $h \in \gamma^*_4$. These linear inequalities naturally extend to all $h \in \CC C_4$ by convexity.

Hence, the conditional inequality \eqref{eq: Matus cond} also extends to conic combinations:
\begin{equation}
  \label{eq: extended cond}
  \langle \Delta_{13|4}, h\rangle = \langle \Delta_{34|1}, h\rangle = 0 \implies \langle \varphi, h\rangle \ge 0.
\end{equation}
If $\CC A_4$ is polyhedral, then its finitely many extreme rays $\alpha_1, \dots, \alpha_r \in \RR^{16}$ provide a polyhedral definition of $\CC C_4 = \Set{ h \in \RR^{16} : \langle \alpha_i, h\rangle \ge 0 }$. Then, by \eqref{eq: extended cond}, the functional $\langle\varphi, \blank\rangle$ is non-negative on the polyhedron $\Set{ h \in \CC C_4 : \langle \Delta_{13|4}, h\rangle = \langle \Delta_{34|1}, h\rangle = 0 }$. By the Farkas~lemma~\cite[Proposition~1.8]{Ziegler} there exist multipliers $\lambda_1, \dots, \lambda_r \ge 0$ and $\mu_1, \mu_2 \in \RR$ such that
\[
  \varphi = \sum_{i=1}^r \lambda_i \alpha_i + \mu_1 \Delta_{13|4} + \mu_2 \Delta_{34|1}.
\]
Let $\ol{\varphi} = \varphi + \lvert\mu_1\rvert \Delta_{13|4} + \lvert\mu_2\rvert \Delta_{34|1}$. Then $\langle \ol{\varphi}, h\rangle \ge \sum_{i=1}^r \lambda_i \langle\alpha_i,h\rangle \ge 0$ for all $h \in \CC C_4$. This shows that $\ol{\varphi}$ is a valid information inequality, contradicting \Cref{lemma: Matus ess cond}.
\end{proof}

\section{The Ingleton and GMM functionals}
\label{sec: 1627}

The goal of this section is a thorough analysis of the family of functionals
\begin{align*}
  [\varphi_k] &\defas [\varphi] \cdot \big( [1:3\mid 4] \cdot [1:4\mid 3] \cdot [2:3\mid 4] \cdot [2:4\mid 3]\big)^k \\
  &= \frac{([13][14][23][24])^{k+1} \cdot [34]^{4k+1}}{[12] \cdot ([3][4][134][234])^{2k+1}},
\end{align*}
for $k \ge 0$ on~$\PD_4$. The Ingleton functional is recovered as the special case~$\varphi = \varphi_0$. Note that $\Sigma[\varphi_k]$ is monotonically increasing in~$k$. We care about two special values: the infimum at $k=0$ and the value $k_*$ at which the functionals become absolutely bounded.

\begin{theorem} \label{thm: varphi_k}
The infimum of $[\varphi_k]$ over $\PD_4$ is given by
\[
  \begin{cases}
    \frac{(4k+4)^{4k+4}}{16 (4k+3)^{4k+3}}, & \text{if $0 \le k < k_*$}, \\
    1, & \text{if $k \ge k_*$},
  \end{cases}
\]
where $k_* \approx 0.59829$ is the unique positive root of the equation $(4k+4)^{4k+4} = 16 (4k+3)^{4k+3}$.
\end{theorem}

The proof \Cref{thm: varphi_k} is technical and we defer it to \Cref{sec: 1627 sym,sec: 1627 permsym,sec: 1627 proof}. Each of these sections indicates a file from our Lean formalization which proves the results of that section. To verify the proof, it suffices to compile the Lean project and to match the definitions and statements of the paper with their formalized forms (\textsf{GMM.Defs} and \textsf{GMM.Infimum}, respectively).
Before presenting the proof, we make some noteworthy observations. Our first is that \Cref{thm: varphi_k} actually answers in the positive the conjecture of Hall and Johnson \cite[Example~2]{HJ}.

\begin{corollary} \label{thm: 1627}
The infimum of the Ingleton functional $[\varphi]$ over $\PD_4$ is $\sfrac{16}{27}$. It is not attained.
\end{corollary}

The next observation requires the following lemma, which uses tools from number theory.

\begin{lemma}
The number $k_*$ from \Cref{thm: varphi_k} is transcendental.
\end{lemma}

\begin{proof}
For simplicity, we apply the affine substitution $u = 4k+3$. Consider the function $f(u) = \frac{(u+1)^{u+1}}{u^u} = (u+1)\!\left(\frac{u+1}{u}\right)^u$. For positive $u$, this is continuous, positive, strictly increasing, and its range is $(1,\infty)$. Hence, there is a unique $u_* > 0$ for which $f(u_*) = 16$. We prove $u_*$ is transcendental, from which it follows that $k_* = \frac14(u_* - 3)$ is also transcendental.

First assume that $u_*$ is rational. Write $u_*=\sfrac{p}{q}$, where $p,q\in\mathbb{N}$ and $\gcd(p,q)=1$. Then
\[
  \left(\frac{p+q}{q}\right)^{(p+q)/q} = 16\left(\frac{p}{q}\right)^{p/q}.
\]
Raising both sides to the power $q$, and then multiplying by $q^{p+q}$, this becomes
\[
  (p+q)^{p+q}=16^q p^p q^q.
\]
Now $\gcd(p+q,p)=\gcd(q,p)=1$ and $\gcd(p+q,q)=\gcd(p,q)=1$, so $\gcd(p+q,pq)=1$. Thus, no prime divisor of $p$ or $q$ can divide the left-hand side, while every prime divisor of $p$ or $q$ divides the right-hand side. Therefore, we must have $p=q=1$, which implies that $u_*=1$. This contradicts $f(1) = 4 \neq 16$, so $u_*$ is irrational.

Thus, if $u_*$ is algebraic (and irrational by the previous paragraph) then $(u_* + 1)/u_*$ is algebraic, nonzero, and different from~$1$. Hence, by the Gelfond--Schneider theorem \cite{Gelfond1934,Schneider1935}, $\big(\frac{u_*+1}{u_*}\big)^{\!u_*}$ is transcendental. Multiplying this by the non-zero algebraic number $u_*+1$ keeps it transcendental, contradicting the fact that it should equal~$16$. Thus, $u_*$, and therefore $k_*$, are transcendental.
\end{proof}

The above lemma gives an obstruction to any rational-algebraic description of the cone of absolutely bounded functionals. The family $\varphi_k$ is a rational affine line of functionals, and \Cref{thm: varphi_k} shows that it enters $\mathcal{A}_4$ exactly at $k=k_*$, a transcendental point.

\begin{corollary}
The convex cones $\CC A_n$, for $n \geq 4$, are not semialgebraic over $\QQ$, i.e., none of them can be described by finitely many polynomial equations and inequalities with rational coefficients.
\end{corollary}

\begin{proof}
Another way to define $k_*$ is as the value
\[
  \inf \{ k \in \RR : \varphi_k \in \CC A_4 \}.
\]
If $\CC A_4$ is semialgebraic over~$\QQ$, this definition can be written as a formula in the first-order theory of $\RR$ with rational coefficients. In this case, the Tarski--Seidenberg theorem and Tarski's transfer principle \cite[Chapter~3]{RAGOpt} imply that $k_*$ is a real algebraic number, which is a contradiction. Since $\CC A_n$ is a coordinate projection of $\CC A_{n+1}$ and projections of semialgebraic sets are semialgebraic, the negative result for $\CC A_4$ extends to all $\CC A_n$ with $n \geq 4$.
\end{proof}

Since $\varphi_k$ is balanced, it is invariant under the scaling $\Sigma \mapsto D \Sigma D$ with a positive definite diagonal matrix~$D$. Hence, it suffices to consider $\PD_4$ matrices with ones on the diagonal, leaving just six unknown off-diagonal entries.
For each fixed $k \in \QQ$ the objective function $[\varphi_k]$ and the constraints defining $\PD_4$ are semialgebraic, so the infimum is computable via cylindrical algebraic decomposition. %
This includes the result of~\Cref{thm: 1627}. However, despite involving only six variables, we were not able to obtain the $\sfrac{16}{27}$ bound purely computationally, %
even after a month of computation in Mathematica. Our proof strategy relies crucially on \emph{symmetry reduction} using positivity-preserving transformations, in particular the Hadamard product \cite[]{Khare2022}.

\begin{proposition} \label{lemma: symmetry reduction}
For any $k \ge 0$ and any $\Sigma \in \PD_4$ there exists a $\Sigma' \in \PD_4$ with nonnegative entries of the form
\begin{gather}
  \label{eq: symmetry reduced}
  \Sigma' = \begin{pmatrix}
    1 & a & b & b \\
    a & 1 & d & d \\
    b & d & 1 & f \\
    b & d & f & 1
  \end{pmatrix}
\end{gather}
such that $\Sigma'[\varphi_k] \le \Sigma[\varphi_k]$.
\end{proposition}

We note that this reduction to four variables is enough for Mathematica's implementation of CAD to verify that $\sfrac{16}{27} < \Sigma'[\varphi]$ for all $\Sigma'$ of the form~\eqref{eq: symmetry reduced}, and thus for all~$\PD_4$. This computation is available in our Zenodo record (\textsf{ingleton-1627.nb}). %

Before turning to the proof of \Cref{thm: varphi_k}, we introduce a family of positive definite matrices whose values of $[\varphi_k]$ approach its infimum. Besides motivating the proof, this family provides the sharp upper bound on the infimum. Hence, for $0\le t<1$, consider the curve of~matrices
\begin{equation}
  \label{eq: inf witness}
  \begin{pmatrix}
    1 & t^\alpha & t & t \\
    t^\alpha & 1 & t & t \\
    t & t & 1 & t^\beta \\
    t & t & t^\beta & 1
  \end{pmatrix},
\end{equation}
where $\alpha=4-\frac{1}{k+1}$ and $\beta= \frac{1}{k+1}$. Its leading principal minors are all positive since $[1]=1$ and
\begin{gather*}
    [12] = 1-t^{2\alpha}, \qquad [123] = (1-t^\alpha)(1+t^\alpha-2t^2),\\
  [1234] = (1-t^\alpha)(1-t^\beta)(1+t^\alpha+t^\beta -3t^4).
\end{gather*}
Moreover, the objective function equals
\[
 [\varphi_k] = \frac{(1+t^{\beta})^{4k+1}}{1+t^{\alpha}}\cdot\frac{1-t^{2}}{1-t^{\alpha}}\cdot\frac{1-t^{2}}{1-t^{\beta}}\cdot\left(\frac{1-t^{2}}{1+t^{\beta}-2t^{2}}\right)^{\!4k+2}.
\]
Applying l'Hôpital's rule on the second, third and fourth factor, we find that
\begin{align*}
    \lim_{t\to 1^-} [\varphi_k] = \lim_{t\to 1^-} \frac{(1+t^{\beta})^{4k+1}}{1+t^{\alpha}}\cdot\frac{2t^{2-\alpha}}{\alpha}\cdot\frac{2t^{2-\beta}}{\beta}\cdot\left(\frac{2}{4-\beta t^{\beta-2}}\right)^{\!4k+2} \!= \frac{(4k+4)^{4k+4}}{16 (4k+3)^{4k+3}}.
\end{align*}
Moreover, we trivially verify that $\operatorname{id}_4[\varphi] = 1$. Hence, it follows that the infimum of $[\varphi_k]$ over $\PD_4$ is bounded above by
\[
\min\bigg\{1,\, \frac{(4k+4)^{4k+4}}{16 (4k+3)^{4k+3}} \bigg\}.
\]
The remainder of this paper is dedicated to proving that this bound is indeed the infimum $[\varphi_k]$ over $\PD_4$, proving \Cref{thm: varphi_k}.

\begin{figure}
\includegraphics[width=0.5\linewidth]{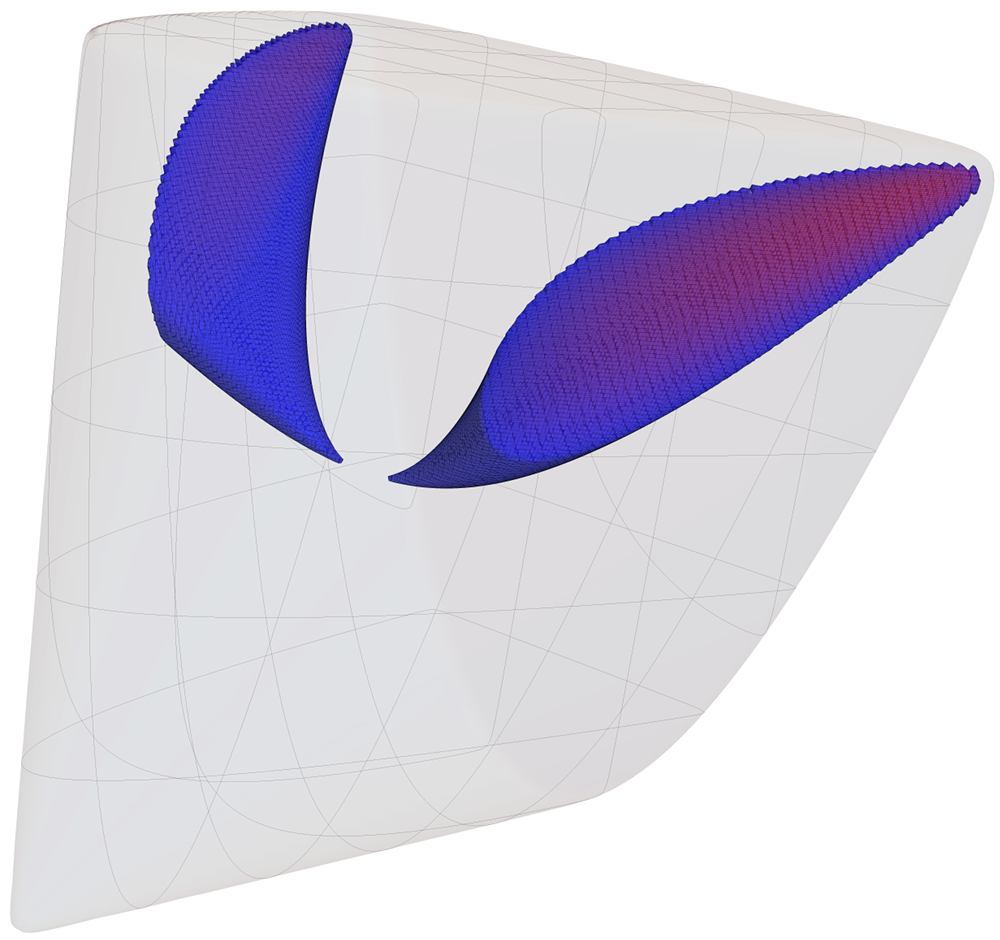}
\caption{\emph{The Wings of Ingleton} --- The gray convex body is the set of all positive definite $4 \times 4$ matrices of the form \eqref{eq: inf witness} in $(t^\alpha,t,t^\beta)$-space. The colored region consists of all points on which $[\varphi] < 1$. Points are colored on a linear gradient according to their value under $[\varphi]$ from blue (close to $1$) to red (close~to~$\sfrac{16}{27}$).}
\label{fig:Wings}
\end{figure}

\subsection{Sign symmetry}\label{sec: 1627 sym}

({\itshape\textsf{Lean: GMM.Reductions.Positivity}}).
We begin the proof of \Cref{lemma: symmetry reduction} with a sign reduction. The goal of this section is to show that, for the purpose of minimizing $\varphi_k$, we may restrict our attention to correlation matrices whose entries $b,c,d,e,f$ are all nonnegative. The argument has two parts. First, the invariance of $\varphi_k$ under diagonal sign changes allows us to assume $f\geq 0$. Second, if the remaining entries do not already have compatible signs, we may replace the matrix by its \emph{Hadamard square}, namely the matrix obtained by squaring each entry. %

\begin{lemma} \label{lemma: normal form 1}
For every $\Sigma \in \PD_4$, there exists $\Sigma' \in \PD_4$ of the form
\begin{equation}
  \label{eq: unit diagonal}
  \Sigma' = \begin{pmatrix}
    1 & a & b & c \\
    a & 1 & d & e \\
    b & d & 1 & f \\
    c & e & f & 1
  \end{pmatrix}
\end{equation}
with $f \geq 0$, and such that $\Sigma[\varphi_k] = \Sigma'[\varphi_k]$ for all $k\geq 0$.
\end{lemma}

\begin{proof}
Since $\varphi_k$ is balanced, it is invariant under the positivity-preserving action $\Sigma \mapsto D \Sigma D$, where $D$ is a positive definite diagonal matrices. Thus, we may assume that $\Sigma$ has unit diagonal, i.e., is of the form \eqref{eq: unit diagonal}. %
The functional $\varphi_k$ and the $\PD$ constraints are also preserved under certain sign flips applied to the entries of the matrix. Consider
\[
  \Sigma'' = \begin{pmatrix}
    1 & \sigma_a a & \sigma_b b & \sigma_c c \\
    \sigma_a a & 1 & \sigma_d d & \sigma_e e \\
    \sigma_b b & \sigma_d d & 1 & \sigma_f f \\
    \sigma_c c & \sigma_e e & \sigma_f f & 1
  \end{pmatrix},
\]
where each $\sigma_x \in \{\pm1\}$. One can check that we may choose $\sigma_a, \sigma_b, \sigma_f$ freely, and the following choices preserve the value of $\varphi_k$ and the $\PD$ constraints:
\begin{gather}
  \label{eq:SignFlips}
  \sigma_c = \sigma_b \sigma_f, \quad
  \sigma_d = \sigma_a \sigma_b, \quad
  \sigma_e = \sigma_a \sigma_b \sigma_f.
\end{gather}
In particular, we may assume $f \geq 0$ by flipping its sign if necessary.
\end{proof}

From now on, we assume that $\Sigma \in \PD_4$ is of the form \eqref{eq: unit diagonal} with $f \geq 0$. The following auxiliary function will be needed:
\begin{equation}\label{eq - gk}
    g_k(x,y,z) := \frac{(1+x^2)^{k+1}(1+y^2)^{k+1}(1+2xyz-x^2-y^2-z^2)^{2k+1}(1+z^2)^{2k}}{(1+2x^2y^2z^2-x^4-y^4-z^4)^{2k+1}}.
\end{equation}

\begin{lemma}\label{lem - g_k}
Suppose that $\Sigma \in \PD_4$ and $f\geq 0$. Then, for all $k \geq 0$,
\[
  0\leq g_k(b,c,f)\leq 1 \quad \text{and}\quad 0\leq g_k(d,e,f) \leq 1.
\]
Moreover, if $bc\leq 0$, then $(1+f^2)g_k(b,c,f)\leq 1$.
\end{lemma}

\begin{proof}
It suffices to prove the bounds for $g_k(b,c,f)$, as the proof for $g_k(d,e,f)$ is analogous. For any $b, c, f \in \RR$ we have $g_k(b,c,f)\geq 0$, since
\[
g_k(b,c,f) = \frac{(1+b^2)^{k+1}(1+c^2)^{k+1}(1+f^2)^{2k} (\Sigma[134])^{2k+1}}{(\Sigma^{\circ 2}[134])^{2k+1}}
\]
and both $\Sigma$ and $\Sigma^{\circ2}$ are positive definite. Now, a direct computation gives
\begin{enumerate}
    \item $\displaystyle \Sigma^{\circ 2}[134]-(1+b^2)(1+c^2)\Sigma[134] = (f-bc)^2\big((1+b^2)(1+c^2)-(f+bc)^2\big)$,

    \item $\displaystyle \Sigma^{\circ 2}[134]-(1+b^2)(1+f^2)\Sigma[134] = (c-bf)^2\big((1+b^2)(1+f^2)-(c+bf)^2\big),$

    \item $\displaystyle \Sigma^{\circ 2}[134]-(1+c^2)(1+f^2)\Sigma[134] = (b-cf)^2\big((1+c^2)(1+f^2)-(b+cf)^2\big)$.
\end{enumerate}
In each case, the last factors are nonnegative by Cauchy--Schwarz, since $b,c,f\leq 1$. Hence, dividing by $\Sigma^{\circ 2}[134]>0$, we obtain
\[
\frac{(1+b^2)(1+c^2)\Sigma[134]}{\Sigma^{\circ 2}[134]} \leq 1,\quad \frac{(1+b^2)(1+f^2)\Sigma[134]}{\Sigma^{\circ 2}[134]} \leq 1 \quad\text{ and }\quad \frac{(1+c^2)(1+f^2)\Sigma[134]}{\Sigma^{\circ 2}[134]} \leq 1 .
\]
Therefore, it follows that
\begin{align*}
    g_k(b,c,f) &= \frac{(1 \hspace{-1.2pt}+\hspace{-1.2pt} b^2)(1 \hspace{-1.2pt}+\hspace{-1.2pt} c^2)\Sigma[134]}{\Sigma^{\circ 2}[134]} \biggl( \frac{(1 \hspace{-1.2pt}+\hspace{-1.2pt} b^2)(1 \hspace{-1.2pt}+\hspace{-1.2pt} f^2)\Sigma[134]}{\Sigma^{\circ 2}[134]}\biggr)^{\!\!k}\! \biggl( \frac{(1 \hspace{-1.2pt}+\hspace{-1.2pt} c^2)(1 \hspace{-1.2pt}+\hspace{-1.2pt} f^2)\Sigma[134]}{\Sigma^{\circ 2}[134]}\biggr)^{\!\!k}\! \leq 1.
\end{align*}

Now, assume that $bc\leq 0$. Then we also have
\[
\Sigma^{\circ 2}[134] \geq (1+b^2)(1+c^2)(1+f^2)\Sigma[134].
\]
Indeed, expanding the terms and grouping them, we find that this is equivalent to
\begin{align*}
     (b^{2}+c^{2}+f^{2}+1)(b^{2}c^{2}f^{2}+b^{2}c^{2}+b^{2}f^{2}+c^{2}f^{2}) \geq 2bcf(b^{2}+1)(c^{2}+1)(f^{2}+1).
\end{align*}
Since $f\geq 0$ and $bc \leq 0$, it is clear that the right-hand side is inferior or equal to $0$ and thus that inequality is satisfied. Therefore, it follows in that case that
\begin{align*}
    g_k(b,c,f) &=  \frac{(1 \hspace{-1.2pt}+\hspace{-1.2pt} b^2)(1 \hspace{-1.2pt}+\hspace{-1.2pt} c^2)\Sigma[134]}{\Sigma^{\circ 2}[134]} \biggl( \frac{(1 \hspace{-1.2pt}+\hspace{-1.2pt} b^2)(1 \hspace{-1.2pt}+\hspace{-1.2pt} f^2)\Sigma[134]}{\Sigma^{\circ 2}[134]}\biggr)^{\!\!k}\! \biggl( \frac{(1 \hspace{-1.2pt}+\hspace{-1.2pt} c^2)(1 \hspace{-1.2pt}+\hspace{-1.2pt} f^2)\Sigma[134]}{\Sigma^{\circ 2}[134]}\biggr)^{\!\!k}\! \leq \frac{1}{1+f^2}
\end{align*}
for all $k\geq 0$.
\end{proof}

For any $\Sigma, \Sigma' \in \PD_n$, Schur's product theorem~\cite[Section~5.2]{MatrixAnalysis} ensures that the \emph{Hadamard product} $\Sigma \circ \Sigma' = [\sigma_{ij} \sigma'_{ij}]_{1\le i,j\le n}$ is also positive definite. In particular, the Hadamard square
\[
  \Sigma^{\circ 2}:=\Sigma\circ \Sigma = \begin{pmatrix}
    1&a^2&b^2&c^2\\a^2&1&d^2&e^2\\b^2&d^2&1&f^2\\c^2&e^2&f^2&1
  \end{pmatrix}
\]
is positive definite and has all entries non-negative.

\begin{lemma}\label{lem - positive_ineq}
Suppose that $\Sigma \in \PD_4$, $f\geq 0$, and that either $bc \leq 0$ or $de\leq 0$. Then, for all $k \geq 0$, we have $\Sigma[\varphi_k] \geq \Sigma^{\circ 2}[\varphi_k]$.
\end{lemma}

\begin{proof}
After simplifying using the elementary identity $1-x^4=(1-x^2)(1+x^2)$, we find that we have $\Sigma[\varphi_k] \geq \Sigma^{\circ 2}[\varphi_k]$ if and only if
\begin{equation}\label{eq - positive_ineq}
    \begin{gathered}
    \frac{1}{(1+2bcf-b^2-c^2-f^2)^{2k+1}(1+2def-d^2-e^2-f^2)^{2k+1}} \\[-6pt]
    \rotatebox{-90}{$\geq$} \\
    \frac{\bigl((1+b^2)(1+c^2)(1+d^2)(1+e^2)\bigr)^{\!k+1}(1+f^2)^{4k+1}}{(1+a^2)(1+2b^2c^2f^2-b^4-c^4-f^4)^{2k+1}(1+2d^2e^2f^2-d^4-e^4-f^4)^{2k+1}}.
    \end{gathered}
\end{equation}
Since $\Sigma$ is positive definite, all parenthesized factors in \eqref{eq - positive_ineq} are positive. Hence, defining $g_k$ as in \eqref{eq - gk}, we have that the inequality \eqref{eq - positive_ineq} is equivalent to
\[
(1+f^2) g_k(b,c,f) \cdot g_k(d,e,f) \leq 1.
\]
But this is always satisfied by \Cref{lem - g_k}, since we may assume without loss of generality that $bc \le 0$. Therefore, we always have $\Sigma[\varphi_k] \geq \Sigma^{\circ 2}[\varphi_k]$ for all $k\geq 0$.
\end{proof}

Combining the previous lemmas yields the following proposition.

\begin{proposition} \label{prop: normal form 2}
For every $\Sigma \in \PD_4$ there exists $\Sigma'$ of the form \eqref{eq: unit diagonal} with $b, c, d, e, f \geq 0$ such that $\Sigma[\varphi_k] \geq \Sigma'[\varphi_k]$ for all $k \geq 0$.
\end{proposition}

\begin{proof}
By \Cref{lemma: normal form 1}, we can assume that $\Sigma$ is of the form \eqref{eq: unit diagonal} with $f \geq 0$. If $bc>0$ and $de>0$, then $b$ and $c$ have the same sign and $d$ and $e$ have the same sign. As seen in the proof of \Cref{lemma: normal form 1}, we may simultaneously flip the signs of $b$ and $c$ to ensure that they are all positive. We may also simultaneously flip the signs of $a$, $d$, and $e$ to ensure that $d, e > 0$. These signs flips do not change the value of $\varphi_k$ or the $\PD$ constraints.
Otherwise $bc\leq 0$ or $de\leq 0$. In~this case \Cref{lem - positive_ineq} shows that $\Sigma' = \Sigma^{\circ2}$ satisfies all claimed properties.
\end{proof}

\subsection{Permutation symmetry}\label{sec: 1627 permsym}

({\itshape\textsf{Lean: GMM.Reductions.Permutation}}).
The previous section reduces our considerations to matrices
\begin{equation}
  \label{eq: normal form 2}
  \Sigma = \begin{pmatrix}
    1 & a & b & c \\
    a & 1 & d & e \\
    b & d & 1 & f \\
    c & e & f & 1
  \end{pmatrix}~ \text{ with $b, c, d, e, f \geq 0$}.
\end{equation}
The aim of this section is to show that we may even assume $b = c$ and $d = e$.

The following is a cylindrical algebraic decomposition of $\PD_4$ (assuming unit diagonals) with variable ordering $f, b, d, c, e, a$:
\begin{gather}
  -1 < f < 1, \quad -1 < b < 1, \quad -1 < d < 1, \nonumber\\
  bf - \sqrt{(1-b^2)(1-f^2)} < c < bf + \sqrt{(1-b^2)(1-f^2)}, \nonumber\\
  df - \sqrt{(1-d^2)(1-f^2)} < e < df + \sqrt{(1-d^2)(1-f^2)}, \nonumber\\
  \frac{bd + ce - bef - cdf - \sqrt{\Delta_1\Delta_2}}{1-f^2} < a < \frac{bd + ce - bef - cdf + \sqrt{\Delta_1\Delta_2}}{1-f^2}, \label{eq: a bounds}
\end{gather}
where
\begin{equation*}
  \Delta_1 \defas \Sigma[134] = 1 + 2bcf - b^2 - c^2 - f^2 \quad\text{and}\quad
  \Delta_2 \defas \Sigma[234] = 1 + 2def - d^2 - e^2 - f^2.
\end{equation*}
If the entries $b, c, d, e, f$ are fixed, then $\varphi_k$ is minimized if the term $\Sigma[12] = 1-a^2$ in the denominator is maximized. Hence, $|a|$ must be as close to zero as the inequality~\eqref{eq: a bounds} allow.

\begin{lemma}\label{lem - PD}
Let $\Sigma \in \PD_4$ of the form \eqref{eq: normal form 2} be arbitrary. Define
\[
  \Sigma' \defas \begin{pmatrix}
  1  & a' & \sqrt{bc} & \sqrt{bc} \\
  a' & 1  & \sqrt{de } &\sqrt{de} \\
  \sqrt{bc} & \sqrt{de} & 1 & f \\
  \sqrt{bc} & \sqrt{de} & f & 1
  \end{pmatrix},
\]
where
\[
  a' \defas \begin{cases}
    0 \qquad & \text{if }~~ 2(bc+de)<1+f,\\
    \frac{2\sqrt{bcde}-\sqrt{\left(1+f-2bc\right)\left(1+f-2de\right)}}{1+f}+\eps & \text{if }~~ 2(bc+de)\geq 1+f.
  \end{cases}
\]
Then $\Sigma'$ is positive definite for all $\eps > 0$ small enough.
\end{lemma}

\begin{proof}
The matrix $\Sigma'$ is positive definite if and only if $\Sigma'[34], \Sigma'[234], \Sigma'[1234] > 0$. The first positivity follows from $\Sigma'[34] = \Sigma[34] > 0$. By the AM-GM inequality $de = \sqrt{d^2e^2} \le \frac{d^2+e^2}{2}$. Thus,
\[
  0 < \Sigma[234] = 1 + 2def - d^2 - e^2 - f^2 \le 1 + 2def - 2de - f^2 = \Sigma'[234].
\]
Define $\delta_1 \defas \Sigma'[134] = (1-f)(1+f-2bc)$ and $\delta_2 \defas \Sigma'[234] = (1-f)(1+f-2de)$. The above argument shows that they are both positive. By specializing \eqref{eq: a bounds} to $\Sigma'$, we see that $\Sigma'[1234] > 0$ is equivalent to
\begin{gather}
  \frac{2(1-f)\sqrt{bcde}-\sqrt{\delta_1\delta_2}}{1-f^{2}} < a' < \frac{2(1-f)\sqrt{bcde}+\sqrt{\delta_1\delta_2}}{1-f^{2}} \label{eq: det Sigma'}.
\end{gather}
There are two cases to distinguish depending on the value of $a'$:
\begin{itemize}
\item If $2(bc+de) < 1+f$, we set $a' = 0$ and \eqref{eq: det Sigma'} is equivalent to $\frac{2(1-f)\sqrt{bcde}-\sqrt{\delta_1\delta_2}}{1-f^{2}} < 0$. This simplifies to $2(bc + de) < 1+f$, which is our assumption.

\item If $2(bc + de) \ge 1+f$ then the definition of $a'$ is such that \eqref{eq: det Sigma'} holds for all $0 < \eps < \frac{2\sqrt{\delta_1\delta_2}}{1-f^2}$. \qedhere
\end{itemize}
\end{proof}

In the following, we define $\Delta_1$, $\Delta_2$, $\delta_1$ and $\delta_2$ as we did previously, namely
\[
\Delta_1 = \Sigma[134],\quad \Delta_2=\Sigma[234],\quad \delta_1=\Sigma'[134]\quad\text{and}\quad  \delta_2=\Sigma'[234].
\]

\begin{lemma}\label{lem - ineq 2}
    Let $\Sigma \in \PD_4$ of the form \eqref{eq: normal form 2} satisfy $2(bc+de)\geq 1+f$. Then
    \[
    |bd+ce-bef-cdf|-\sqrt{\Delta_1\Delta_2} \geq 2(1-f)\sqrt{bcde}-\sqrt{\delta_1\delta_2} \geq 0.
    \]
\end{lemma}
\begin{proof}
    A direct computation yields $\Delta_1 = \delta_1-(b-c)^2$ and $\Delta_2 = \delta_2 -(d-e)^2$. Hence, an application of Cauchy--Schwarz reveals that
    \begin{align*}
        \sqrt{\Delta_1}\cdot \sqrt{\Delta_2} + |b-c|\cdot|d-e| \leq \sqrt{\Delta_1+(b-c)^2}\sqrt{\Delta_2+(d-e)^2} = \sqrt{\delta_1\delta_2}
    \end{align*}
    Moreover, by the triangle inequality
    \begin{align*}
        (be+cd)(1-f) &\leq |(be+cd)(1-f)+(b-c)(d-e)|+|b-c||d-e| \\
        &= |bd+ce-bef-cdf| + |b-c||d-e|.
    \end{align*}
    Therefore, it follows that
    \begin{align*}
        |bd+ce-bef-cdf|-\sqrt{\Delta_1\Delta_2} & \geq (be+cd)(1-f) - |b-c||d-e| - \sqrt{\Delta_1\Delta_2} \\
        & \geq (be+cd)(1-f)-\sqrt{\delta_1\delta_2}.
    \end{align*}
    The first inequality then follows from the AM-GM inequality, since $be+cd \geq  2\sqrt{bcde}.$ The second one is almost immediate, as it is equivalent to $2(bc+de)\ge1+f$, which is our hypothesis.
\end{proof}

\begin{corollary}\label{cor - min_a}
    Let $\Sigma \in \PD_4$ of the form \eqref{eq: normal form 2} satisfy $2(bc+de) \geq 1+f$. Moreover, let
    \[
    L:=\frac{bd+ce-bef-cdf-\sqrt{\Delta_1\Delta_2}}{1-f^{2}} \quad\text{ and }\quad U:=\frac{bd+ce-bef-cdf+\sqrt{\Delta_1\Delta_2}}{1-f^{2}}.
    \]
    Then $L$ and $U$ have the same sign, and the value of $a\in[L,U]$ minimizing $a^2$ is
    \[
     a= \frac{bd+ce-bef-cdf-\epsilon\sqrt{\Delta_1\Delta_2}}{1-f^{2}},
    \]
    where $\epsilon = \operatorname{sign}(bd+ce-bef-cdf)$. In that case, we have
    \[
    \epsilon \cdot a = \frac{|bd+ce-bef-cdf|-\sqrt{\Delta_1\Delta_2}}{1-f^{2}}.
    \]
\end{corollary}
\begin{proof}
    Note that $L$ and $U$ have the same sign if and only if $LU\geq 0$. We easily compute via the elementary identity $(x-y)(x+y)=x^2-y^2$ that $LU\geq 0$ if and only if
    \[
    (bef+cdf-bd-ce)^{2}\geq (1+2bcf-b^{2}-c^{2}-f^{2})(1+2def-d^{2}-e^{2}-f^{2}).
    \]
    Hence, it follows from \Cref{lem - ineq 2} that $LU\geq0$, which establishes our first claim. Now, since $L$ and $U$ have the same sign, any number $a$ such that $L\leq a\leq U$ satisfies
    \[
    \min\{L^2,U^2\} \le a^2 \le \max\{L^2,U^2\}.
    \]
    Hence, the smallest $a^2$ such that $L<a<U$ is $\min\{L^2,U^2\}$. Observe that
    \[
    U^2-L^2 = 2(bd+ce-bef-cdf)\sqrt{\Delta_1\Delta_2}.
    \]
    Since $\Delta_1\Delta_2>0$, the sign of $U^2-L^2$ only depends on the sign of $bd+ce-bef-cdf$. In particular, the minimum is realized by $L^2$ if $bd+ce-bef-cdf>0$, and by $U^2$ otherwise, that is
    \begin{equation*}
        a= \frac{bd+ce-bef-cdf-\epsilon\sqrt{\Delta_1\Delta_2}}{1-f^{2}}. \qedhere
    \end{equation*}
\end{proof}

\begin{proposition}\label{prop - full reduction}
    Let $\Sigma \in \PD_4$ of the form \eqref{eq: normal form 2} be arbitrary, and define
    \[
    \Sigma':=\begin{pmatrix}
        1&a'&\sqrt{bc}&\sqrt{bc}\\a'&1&\sqrt{de}&\sqrt{de}\\\sqrt{bc}&\sqrt{de}&1&f\\\sqrt{bc}&\sqrt{de}&f&1
    \end{pmatrix},
    \]
    where
    \[
    a':=\begin{cases}
        0 \qquad &\text{if }~~ 2(bc+de)<1+f;\\
        \frac{2\sqrt{bcde}-\sqrt{\left(1+f-2bc\right)\left(1+f-2de\right)}}{1+f}+\varepsilon &\text{if }~~ 2(bc+de)\geq 1+f.
    \end{cases}
    \]
    Then, $\Sigma'\in \PD_4$ and, for all $k\geq 0$, $\Sigma[\varphi_k] \geq \Sigma'[\varphi_k]$ whenever $\varepsilon>0$ is sufficiently small. %
\end{proposition}
\begin{proof}
    Since $\Sigma$ is positive definite, \Cref{lem - PD} ensures that $\Sigma'$ is also positive definite. Hence, $\varphi_k$ is well defined for $\Sigma'$ and the result follows if we can show that $\Sigma[\varphi_k] \geq \Sigma'[\varphi_k]$ for all $\varepsilon>0$ sufficiently small. Therefore, first observe that, by the AM-GM inequality,
    $$
    \Delta_1 = 1+2bcf-b^{2}-c^{2}-f^{2} \leq 1+2bcf-2bc-f^{2} %
    = \delta_1.
    $$
    Hence, it follows that
    \begin{align*}
        \frac{(1-b^{2})^{k+1}(1-c^{2})^{k+1}}{\Delta_1^{2k+1}} &= \frac{1}{\Delta_1^{k}}\biggl(1+\frac{(bc-f)^{2}}{\Delta_1}\biggr)^{\!k+1} \!\geq  \frac{1}{\delta_1^{k}}\biggl(1+\frac{(bc-f)^{2}}{\delta_1}\biggr)^{\!k+1}\! = \frac{(1-bc)^{2k+2}}{\delta_1^{2k+1}}.
    \end{align*}
    Similarly,
    \begin{equation*}
        \frac{(1-d^{2})^{k+1}(1-e^{2})^{k+1}}{\Delta_2^{2k+1}} \geq \frac{(1-de)^{2k+2}}{\delta_2^{2k+1}}.
    \end{equation*}

For fixed values of $b,c,d,e,f$, the function $\varphi_k$ is always increasing relative to $a^2$. Hence, to minimize the function, we need $a^2$ to be as small as possible. If $2(bc+de)<1+f$, then
\begin{align*}
    \Sigma[\varphi_k] &= \frac{\bigl((1-b^2)(1-c^2)(1-d^2)(1-e^2)\bigr)^{\!k+1}(1-f^2)^{4k+1}}{(1-a^2)\Delta_1^{2k+1}\Delta_2^{2k+1}} \\
    &\geq \frac{(1-bc)^{2k+2}(1-de)^{2k+2}(1-f^2)^{4k+1}}{\delta_1^{2k+1}\delta_2^{2k+1}}
    \,=\, \Sigma'[\varphi_k].
\end{align*}
Now suppose that $2(bc+de)\geq 1+f$. Since $\Sigma\in \PD_4$, we have $\det(\Sigma)>0$, which means that
\[
\frac{bd+ce-bef-cdf-\sqrt{\Delta_1\Delta_2}}{1-f^{2}}<a<\frac{bd+ce-bef-cdf+\sqrt{\Delta_1\Delta_2}}{1-f^{2}}.
\]
\Cref{cor - min_a} thus implies that the infimum of $a^2$ in this open interval is given by
\[
a_0^2:= \left(\frac{bd+ce-bef-cdf-\epsilon\sqrt{\Delta_1\Delta_2}}{1-f^{2}} \right)^{\!\!2},
\]
where $\epsilon= \operatorname{sign}(bd+ce-bef-cdf)$, and it is never attained. Hence, we add a small perturbation $\epsilon\varepsilon$ to $a_0$ to ensure that $|a_0+\epsilon\varepsilon| \leq |a|$ and that the determinant of $\Sigma$ with $a$ replaced by $a_0+\epsilon\varepsilon$ is positive. In that case, \Cref{lem - ineq 2} yields
\[
\epsilon (a_0 + \epsilon\varepsilon) =\frac{|bd+ce-bef-cdf|-\sqrt{\Delta_1\Delta_2}}{1-f^{2}} + \varepsilon \geq \frac{2(1-f)\sqrt{bcde}-\sqrt{\delta_1\delta_2}}{1-f^2} + \varepsilon = a' > 0.
\]
Therefore, for $\varepsilon>0$ small enough, we have $a^2 \geq (a_0+\epsilon\varepsilon)^2 \geq a'^2$, which finally implies that
\begin{align*}
    \Sigma[\varphi_k] &= \frac{\bigl((1-b^2)(1-c^2)(1-d^2)(1-e^2)\bigr)^{\!k+1}(1-f^2)^{4k+1}}{(1-a^2)\Delta_1^{2k+1}\Delta_2^{2k+1}} \\
    &\geq \frac{(1-bc)^{2k+2}(1-de)^{2k+2}(1-f^2)^{4k+1}}{(1-a'^2)\delta_1^{2k+1}\delta_2^{2k+1}}
    \,=\, \Sigma'[\varphi_k]. \qedhere
\end{align*}
\end{proof}

\subsection{Analytic bounds}
\label{sec: 1627 proof}

({\itshape\textsf{Lean: GMM.Infimum}}).
In this section, we finally complete the proof of \Cref{thm: varphi_k}. \Cref{prop - full reduction} ensures that it suffice to consider the minimum of $[\varphi_k]$ over matrices $\Sigma'\in \PD_4'$, where $\PD_4'$ is the set of $4\times 4$ positive definite matrices of the form
\[
\Sigma'=\begin{pmatrix}
    1&a&u&u\\a&1&v&v\\u&v&1&f\\u&v&f&1
\end{pmatrix} \quad\text{ with }\quad a = \begin{cases}
        0 \qquad &\text{if }~~ 2(u^2+v^2)<1+f;\\
        \frac{2uv-\sqrt{(1+f-2u^2)(1+f-2v^2)}}{1+f}+\varepsilon &\text{if }~~ 2(u^2+v^2)\geq 1+f.
    \end{cases}
\]
and $u,v,f\geq 0$. Under this assumption, we may define for simplicity
\[
s:=\frac{1-f}{2}\in(0,1/2],\qquad x:=\frac{u^2}{1-s-u^2}
\quad\text{ and }\quad
y:=\frac{v^2}{1-s-v^2}.
\]
With these changes of variables, we find that $\Sigma'\in \PD_4'$ if and only if $\Sigma'[234],\Sigma'[134]>0$, that is $u^2,v^2<1-s$. Consequently, we have $x,y>0$. Moreover, the condition $2(u^2+v^2)<1+f$ simplifies to $xy<1$, so that the objective function becomes simply, as $\varepsilon\to 0$,
\[
\Sigma'[\varphi_k] =\frac{(1+sx)^{2k+2}(1+sy)^{2k+2}}{4s(1-s)\,h(x,y)},
\]
where
\[
h(x,y) = \begin{cases}
        (1+x)(1+y) \qquad &\text{if }~~ xy<1;\\
        (\sqrt{x}+\sqrt{y})^{2} &\text{if }~~ xy\geq 1.
    \end{cases}
\]
Hence, the reduced problem naturally splits in two cases. On the region $xy\geq 1$, the sharp lower bound is the boundary value already obtained from the curve \eqref{eq: inf witness}. On the complementary region $xy<1$, the relevant lower bound depends on whether $k\leq 1/2$ or $k\geq 1/2$. The next propositions establish these two regional estimates, which are then assembled in the proof of \Cref{thm: varphi_k}.

\begin{proposition}\label{prop - main2}
    The infimum of $[\varphi_k]$ over all $\Sigma'\in \PD_4'$ such that $xy\geq 1$ is $\frac{(4k+4)^{4k+4}}{16(4k+3)^{4k+3}}$.
\end{proposition}
\begin{proof}
Let $a:=2k+2$, $r:=s\sqrt{xy}$ and $t:=\frac{s}{4}(\sqrt{x}+\sqrt{y})^2$. By AM-GM, we have $t\ge r$. Moreover,
\begin{equation}\label{eq - function_min}
    \Sigma'[\varphi_k]=\frac{(1+sx)^a(1+sy)^a}{4s(1-s)(\sqrt{x}+\sqrt{y})^2}
=
\frac{\bigl(4t+(1-r)^2\bigr)^a}{16(1-s)t}.
\end{equation}
We seek to minimize \eqref{eq - function_min} under the condition $t\ge r$. For fixed $t$, the quantity
$4t+(1-r)^2$ is minimized by taking $r$ as close as possible to $1$. If $0<t\le 1$, this gives $r=t$. Hence,
\[
\Sigma'[\varphi_k]=
\frac{\bigl(4t+(1-r)^2\bigr)^a}{16(1-s)t} \geq \frac{\bigl(4t+(1-t)^2\bigr)^a}{16(1-s)t} = \frac{(1+t)^{2a}}{16(1-s)t}.
\]
If $t\ge 1$, the minimum is obtained at $r=1$, giving
\[
\Sigma'[\varphi_k]=
\frac{\bigl(4t+(1-r)^2\bigr)^a}{16(1-s)t} \geq \frac{(4t)^{a-1}}{4(1-s)} \geq \frac{4^{a-2}}{1-s},
\]
which is exactly the value of $\frac{(1+t)^{2a}}{16(1-s)t}$ at $t=1$. It thus remains to minimize the function
\[
H(t)=\frac{(1+t)^{2a}}{16(1-s)t},
\qquad 0<t\le 1.
\]
A direct computation yields $\frac{H'(t)}{H(t)} = \frac{(2a-1)t-1}{t(1+t)}$, which implies that the minimum is attained at $t=\frac{1}{2a-1} \in (0,1)$. Moreover, we have $1-s\leq  1$ for $0<s\leq 1/2$. Therefore, it follows that
\[
\inf_{\Sigma'\in\PD_4',\,xy\geq1} \Sigma'[\varphi_k]
=
\frac{\big(1+\frac{1}{2a-1}\big)^{2a}}
{16\cdot \frac{1}{2a-1}}
=
\frac{(2a)^{2a}}{16(2a-1)^{2a-1}} = \frac{(4k+4)^{4k+4}}{16(4k+3)^{4k+3}}.\qedhere
\]
\end{proof}

\begin{lemma}\label{lem - keyineq}
If $x,y> 0$, $0<s\le 1/2$, and $xy<1$, then $\Sigma'[\varphi_{1/2}] \geq 1$.
\end{lemma}
\begin{proof}
Let $r:=\sqrt{xy}$ and $t:=(1+x)(1+y)$. Then $0< r<1$ and, by the AM-GM inequality, we have both $t\geq (1+r)^2$ and
\begin{equation}\label{eq - AMGM2}
    \begin{aligned}
    (1+sx)^{2}(1+sy)^{2} = (st+(1-s)(1-sr^2))^2 \ge 4s(1-s)t(1-sr^2).
\end{aligned}
\end{equation}
Therefore, it follows that
\begin{align*}
    \Sigma'[\varphi_{1/2}] &= \frac{(st+(1-s)(1-sr^2))^2}{4s(1-s)t} (st+(1-s)(1-sr^2))\\
    &\geq (1-sr^2)\bigl(s(1+r)^2+(1-s)(1-sr^2)\bigr) \\
    &= 1+sr\bigl((1-r)(2+r)+r^{2}(1-2s)+rs(1-r^{2}s)\bigr) \\
    &\geq 1. \qedhere
\end{align*}
\end{proof}

\begin{proposition}\label{prop - inf 1}
    The infimum of $\varphi_k$ over all $\Sigma'\in \PD_4'$ such that $xy<1$ is $1$ if $k\geq 1/2$.
\end{proposition}
\begin{proof}
Since $s,x,y\geq 0$ and since $k\geq 1/2$, we have
\[
\Sigma'[\varphi_{k}] = \frac{(1+sx)^{2k+2}(1+sy)^{2k+2}}{4s(1-s)(1+x)(1+y)} \geq \frac{(1+sx)^{3}(1+sy)^{3}}{4s(1-s)(1+x)(1+y)} = \Sigma'[\varphi_{1/2}].
\]
It thus directly follows from \Cref{lem - keyineq} that $\Sigma'[\varphi_{k}] \geq 1$. To show that the bound is indeed realized, simply choose $x=y=\frac{1-2s}{s}$ to obtain $\Sigma'[\varphi_{k}] = 2s(2-2s)^{4k+1} \to 1$ as $s\to 1/2$.
\end{proof}

\begin{lemma}\label{lem - keyineq2}
    For every $x,y\ge 0$, $0<s\le 1/2$ such that $xy<1$, we have $\Sigma'[\varphi_0] \geq 3/4$.
\end{lemma}
\begin{proof}
    First suppose that $s\ge 1/4$, and observe that for every $z\ge 0$,
    \[
    (1+sz)^2-4s(1-s)(1+z) = \bigl(sz-(1-2s)\bigr)^2\ge 0.
    \]
    Hence, $\frac{(1+sz)^2}{1+z}\ge 4s(1-s).$ Applying this bound with $z=x$ and $z=y$, we get
    \[
    \Sigma'[\varphi_0] = \frac{(1+sx)^2(1+sy)^2}{4s(1-s)(1+x)(1+y)}\ge 4s(1-s) \geq \frac34.
    \]
    Moreover, if $0<s<1/4$, then inequality \eqref{eq - AMGM2} also ensures that
    \begin{align*}
        \Sigma'[\varphi_0] &= \frac{(1+sx)^2(1+sy)^2}{4s(1-s)(1+x)(1+y)} \geq 1-sxy \geq 1-s \geq \frac{3}{4}.\qedhere
    \end{align*}
\end{proof}

\begin{proposition}\label{prop - better}
    The infimum of $\varphi_k$ over all $\Sigma'\in \PD_4'$ such that $xy< 1$ is bounded below by $\frac{(4k+4)^{4k+4}}{16(4k+3)^{4k+3}}$ if $0\leq k\leq 1/2$.
\end{proposition}
\begin{proof}
Observe that\vspace{-4pt}
\[
\Sigma'[\varphi_k] = \Sigma'[\varphi_{k_1}]^{\frac{k_2-k}{k_2-k_1}}\cdot \Sigma'[\varphi_{k_2}]^{\frac{k-k_1}{k_2-k_1}}
\qquad (k_1<k<k_2).
\]
\Cref{lem - keyineq,lem - keyineq2} then imply that
\begin{align*}
    \Sigma'[\varphi_k] = \Sigma'[\varphi_{0}]^{\frac{1/2-k}{1/2-0}}\cdot \Sigma'[\varphi_{1/2}]^{\frac{k-0}{1/2-0}} \geq \left( \frac{3}{4}\right)^{\!1-2k}\!\left( 1\right)^{2k} = \left( \frac{4}{3}\right)^{\!2k-1}.
\end{align*}
Let $x=4k+3$, so that $3\leq x \leq 5$. We seek to show $\left( \frac{4}{3}\right)^{2k-1} \geq \frac{(4k+4)^{4k+4}}{16(4k+3)^{4k+3}}$, that is
\begin{equation}\label{eq - equiv_form}
    G(x):=(x-5)\log\bigl(2/\sqrt{3}\bigr)-(x+1)\log(x+1)+x\log(x)+\log(16) \geq 0.
\end{equation}
A direct computation yields
\[
G'(x)= \log\!\left(\frac{2}{\sqrt{3}}\right)-\log\!\left(1+\frac1x\right) \leq  \log\!\left(\frac{2}{\sqrt{3}}\right)-\log\!\left(1+\frac15\right) = \log\!\left(\frac{5}{3\sqrt{3}}\right)<0.
\]
Therefore, $G(x) \geq G(5)=\log\!\left(\frac{3125}{2916}\right)>0$ on $[3,5]$, which is precisely \eqref{eq - equiv_form}.
\end{proof}

We may now finally prove \Cref{thm: varphi_k}, stating that the infimum of $\varphi_k$ over all $\Sigma'\in \PD_4'$ is
\[
\begin{cases}
    \frac{(4k+4)^{4k+4}}{16(4k+3)^{4k+3}} \quad &\text{if }~ 0 \leq k < k_*; \\
    1 \quad &\text{if }~ k \geq k_*,
\end{cases}
\]
where $k_* \approx 0.59829$ is the unique positive root of the equation $(4t+4)^{4t+4}=16(4t+3)^{4t+3}$.

\begin{proof}[Proof of \Cref{thm: varphi_k}]
    For $0 \leq k \leq 1/2$, \Cref{prop - better} shows that the infimum of $\Sigma'[\varphi_k]$ when $xy<1$ is larger than the infimum of $\Sigma'[\varphi_k]$ when $xy\geq1$, which is
    \[
    \frac{(4k+4)^{4k+4}}{16(4k+3)^{4k+3}}.
    \]
    Moreover, for $k\geq 1/2$, the infimum when $xy\geq 1$ is once again $\frac{(4k+4)^{4k+4}}{16(4k+3)^{4k+3}}$, while \Cref{prop - inf 1} ensures that the infimum when $xy<1$ is $1$. Therefore, if $k\geq 1/2$, the infimum of $\Sigma'[\varphi_k]$ is
    \[
    \min\!\left\{1, \frac{(4k+4)^{4k+4}}{16(4k+3)^{4k+3}} \right\}.
    \]
    It is easily verified that $\frac{(4k+4)^{4k+4}}{16(4k+3)^{4k+3}}$ is strictly increasing for all $k\geq 0$. Hence, the infimum is realized by this function if $k\leq k_*$ and by $1$ otherwise, which finally completes the proof.
\end{proof}

\begin{remark}[On the limiting value]
\Cref{thm: varphi_k}, shows that the infimum of $[\varphi]$ is approached along a curve of correlation matrices converging to a rank-one matrix. The Curve Selection Lemma \cite[Theorem 3.19]{BasuPollackRoy2006} guarantees that one may take such a curve to have Puiseux series entries with real-algebraic coefficients, as in~\eqref{eq: inf witness}. Along this curve, every principal minor has a Puiseux expansion with algebraic leading coefficient. The exponent of the first nonzero term determines whether the corresponding ratio tends to 0, to $+\infty$, or to a finite nonzero limit. In the last case, however, the limit depends on the leading coefficients themselves, and is therefore a priori only algebraic. Tropical methods as developed in \cite{ARVY} indicate whether a given functional is bounded, but do not carry information about the infimum. Thus, we cannot offer a satisfying explanation for why the infimum of $[\varphi]$ is the rational number $\sfrac{16}{27}$.
\end{remark}

\section*{Acknowledgements}

\setlength{\intextsep}{5pt}%
\setlength{\columnsep}{5pt}%
\begin{wrapfigure}{R}{0.165\linewidth}
\vspace{-.5\baselineskip}%
\centering%
\href{https://doi.org/10.3030/101110545}{%
\includegraphics[width=0.9\linewidth]{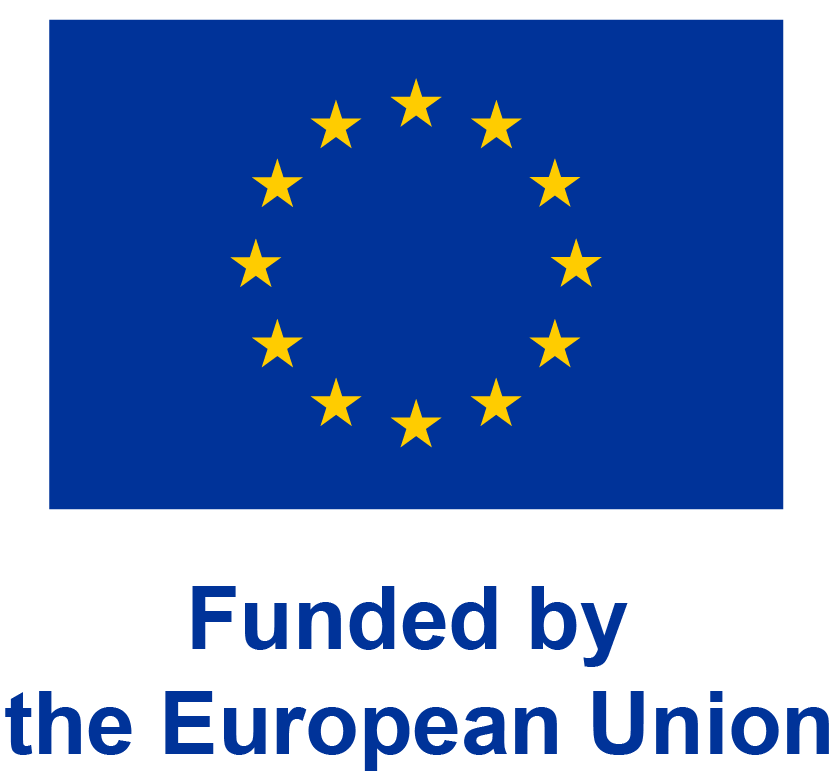}%
}
\end{wrapfigure}
The authors wish to thank the organizers and sponsors of the 12th Heidelberg Laureate Forum where this collaboration was started.
We are indebted to Thomas Kahle for helping with the Lean formalization of the proof of \Cref{thm: varphi_k}. %
T.B. was funded by the European Union's Horizon 2020 research
and innovation programme under the Marie Skłodowska-Curie grant agreement
No.~101110545.

\printbibliography

\end{document}